\documentclass{amsart}
\usepackage{diagrams}
\usepackage{amssymb}
\usepackage{dsfont}
\usepackage{epsfig}
\usepackage{graphicx, color}

\newcommand{\NN}{{\mathbb N}}
\newcommand{\ZZ}{{\mathbb Z}}
\newcommand{\QQ}{{\mathbb Q}}

\renewcommand{\geq}{\geqslant}

\def\B{{\mathcal B}}

\def\fn{{f_{/n}}}

\numberwithin{equation}{section}

\newtheorem{theo}{Theorem}
\newtheorem{prop}[theo]{Proposition}
\newtheorem{coro}[theo]{Corollary}
\newtheorem{lem}[theo]{Lemma}
\theoremstyle{definition}
\newtheorem{defi}[theo]{Definition}
\newtheorem{rema}[theo]{Remark}
     
\begin{document}

\title{Minimal polynomial dynamics on the set of 3-adic integers}
\author{Fabien Durand}
\author{Fr\'ed\'eric Paccaut}
\address{
Universit\'e de Picardie Jules Verne\newline
Laboratoire Ami\'enois de Math\'ematiques Fondamentales et Appliqu\'ees\newline
CNRS-UMR 6140\newline
33 rue Saint Leu\newline
80039 Amiens Cedex 01\newline
France.}
\email{fabien.durand@u-picardie.fr}
\email{frederic.paccaut@u-picardie.fr}

\begin{abstract}
In this paper are characterized the polynomials, in terms of their coefficients, that have all their orbits dense in the set of 3-adic integers $\ZZ_3$.
\end{abstract}

\maketitle

\section{Introduction}

The study of dynamical systems on $\ZZ_p$ or $\QQ_p$ is now developing a lot \cite{An1,BS,CP,FLWZ,FLYZ}. It is sometimes guided by or leading to 
applications in physics, cognitive science and cryptography \cite{An2,KN}.
For example, in the theory of pseudo random number generators, it is usefull to have a map, defined on the
integers, giving rise to large cycles
modulo $n$ for a given integer $n$. When $n$ is a power of a prime, good candidates are minimal maps in the set of $p$-adic integers (that is maps with all
their orbits
dense). Namely, these minimal maps, whenever polynomial, have only one cycle with maximal length modulo $p^n$ for every $n$. The point is:
how to find such a minimal map or
decide if a given polynomial map is minimal ?

There is a complete characterization of minimal polynomial maps (with a different vocabulary) for $p=2$ in \cite{La} and necessary conditions on the Mahler coefficients
of 1-Lipschitz maps to be minimal are given in \cite{An1}. Complete results also exist for 
affine maps (see \cite{CP} and \cite{FLYZ} where the dynamics are also studied when non minimal) and for polynomials of degree 2 for all prime $p$ \cite{Kn}.
But a complete description of minimal polynomials in $\ZZ_p$
in terms of their coefficients seems to be a much harder task.
Here we give this description for $p=3$.

First, we give definitions concerning $p$-adic analysis and dynamical systems. In Section 3, we develop the main properties of the dynamics of a compatible map
(a map that preserves congruences) on the set of $p$-adic integers. In particular, we prove that an onto compatible map is minimal if and only if it is conjugated to
the odometer in base $p$. Section 4 is devoted to the characterization of minimality for polynomials. In this section, we have reorganized ideas of \cite{DZ} and
rewritten a proof from \cite{La} so as to develop a strategy to deal with $\ZZ_2$ and $\ZZ_3$. We hope this strategy will also apply to $\ZZ_p$ for $p>3$.

\section{Some definitions}

\subsection{$p$-adic integers and compatible maps}
Let $p$ be a prime number.
We endow $\ZZ / p^n \ZZ$ with the discrete topology and $\prod_{n\in \NN} \ZZ/p^{n} \ZZ$ with the product topology.
Let $\varphi_n : \ZZ/p^{n+1} \ZZ \to  \ZZ/p^{n} \ZZ$, $n\in \NN$, be the canonical homomorphisms.
We call $\ZZ_p$ the projective limit of $(\ZZ/p^{n} \ZZ , \varphi_n )$.
This limit can be seen as

$$
\ZZ_p = \left\{ (x_n)_n \in \prod_{n\in \NN} \ZZ/p^{n} \ZZ  ; \varphi_n (x_{n+1}) = x_n \right \}
$$

which is a compact subset of $\prod_{n\in \NN} \ZZ/p^{n} \ZZ$ (see \cite{Ro}).
The canonical projections are $\pi_n : \ZZ_p \to \ZZ/p^{n} \ZZ$.
We of course have $\varphi_{n} \circ \pi_{n+1} = \pi_{n}$.

Let $x=(x_n)\in \ZZ_p$ and $y=(y_n)\in \ZZ_p$.
Then $(x_n+y_n)_n$ and $(x_n.y_n)_n$ are also  elements of $ \ZZ_p$;
this defines the addition $x+y$ and the multiplication $x.y$ in $ \ZZ_p$.
It can be checked that with these operations $ \ZZ_p$ is a topological ring.

As a topological space $\ZZ_p$, is a Cantor set (a compact metric space with no isolated point and a countable base of its topology consisting of clopen sets).
More precisely, one can prove that it is homeomorphic to $\prod_{n\in \NN} \{ 0,1,\dots , p-1\}$ endowed with the infinite product of the discrete topologies.

Through this homeomorphism, a point $x=(x_n)_n$ in $\ZZ_p$ can be represented by $x=x_0'+x_1'p+x_2'p^2+\ldots$ with $x_i'\in\{0,1,\ldots,p-1\}$ and 
$x_i=x_0'+x_1'p+\ldots+x_i'p^i$. In this way, $\ZZ/p^n\ZZ$ can be viewed as a subset of $\ZZ_p$ :
  $$\ZZ/p^n\ZZ=\{x\in\ZZ_p, \forall i\geq n, x_i'=0\}.$$
Thereofore, when no confusion is possible, we will consider an element $x\in\ZZ/p^n\ZZ$ as an element $x=x_0'+x_1'p+x_2'p^2+\ldots+x_i'p^i$ of $\ZZ_p$ with $i\le n-1$.
By definition, $\vert x\vert_p=p^{-\min\{i,x_i'\neq 0\}}$ is the $p$-adic norm of $x$ (with $\vert 0\vert_p=0$). Each set $x+p^k\ZZ_p$ ($x\in\ZZ_p, k\in{\NN}$) is
a clopen ball of radius $p^{-k}$ and for each $k$, $\ZZ_p$ is the union of $p^k$ balls of radius $p^{-k}$.

The Haar measure, the unique translation invariant probability measure on $\ZZ_p$, gives measure $p^{-k}$ to any ball of radius $p^{-k}$. This measure will
be denoted by $\mu_p$.

The following lemma will be extensively used in this paper. 
The proof is left to the reader.

\begin{lem}
\label{compatible}
Let $f:\ZZ_p \to \ZZ_p$ be a function.
Then, the following are equivalent.

\begin{enumerate}
\item
For every $n$, if $\pi_n (x) = \pi_n (y) $ then  $\pi_n (f(x)) = \pi_n (f(y))$;
\item
For every $x$ and $y$ in ${\ZZ}_p$, $\vert f(y)-f(x)\vert_p \le\vert y-x\vert_p$;
\item
For every $n$ and all $x\in \ZZ_p$, $f(x+p^n \ZZ_p ) \subset f(x) + p^n \ZZ_p $;
\item
For every $n$, there exists a unique map $\fn : \ZZ/p^n \ZZ \to \ZZ /p^n \ZZ$ satisfying

\begin{equation}
\pi_n \circ f = \fn \circ \pi_n .
\end{equation}
\end{enumerate}
Moreover we have the following other commuting relation:

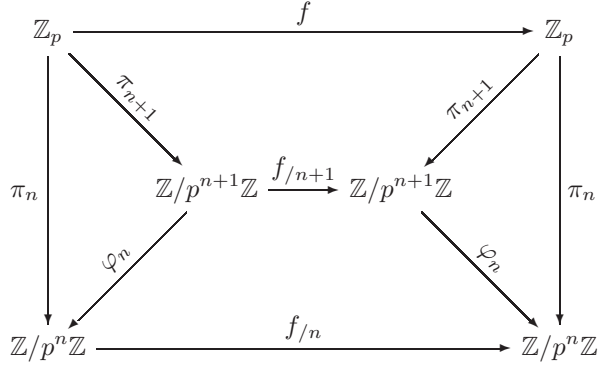
\begin{equation}
\label{proj-fn}
\fn \circ \varphi_{n} = \varphi_{n} \circ f_{/n+1} .
\end{equation}
\end{lem}

\begin{diagram}
{\ZZ}_p               &                    &                      &      \rTo^f         &                        &                       & {\ZZ}_p       \\
                      &  \rdTo^{\pi_{n+1}} &                      &                     &                        &    \ldTo^{\pi_{n+1}}  &               \\
\dTo^{\pi_n}          &                    &   {\ZZ}/p^{n+1}{\ZZ} &     \rTo^{f_{/n+1}}  &   {\ZZ}/p^{n+1}{\ZZ}   &                      & \dTo_{\pi_n}  \\
                      &  \ldTo^{\varphi_n} &                      &                     &                        &    \rdTo^{\varphi_n}  &               \\
{\ZZ}/p^n{\ZZ}        &                    &                      &       \rTo^{\fn}    &                        &                       & {\ZZ}/p^n{\ZZ} \\
\end{diagram}

\begin{defi}
A {\it compatible map} $f : \ZZ_p \to \ZZ_p$ is 
a map satisfying one of the assertions of Lemma \ref{compatible}.
\end{defi}

Note that compatible maps are exactly 1-Lipschitz maps. This terminology can be found in \cite{An1}.
In the sequel, whenever $f$ is a compatible map, $\fn$ will always denote the map defined
in Lemma \ref{compatible}.

\begin{prop}
Any polynomial whose coefficients belong to ${\ZZ}_p$ is compatible.
\end{prop}

\subsection{Dynamical systems}

Let $X$ be a compact metric space and $T:X\to X$  be a homeomorphism. The couple $(X,T)$ is called a topological dynamical system. Let $\B$ be the Borel
$\sigma$-algebra of $X$ and $\mu$ a probability measure defined on $\B$. 

A set $A\in\B$ is called invariant by $T$ if $T^{-1}A=A$.
We say that $\mu$ is invariant by $T$ if, for every set $A\in\B$, $\mu(T^{-1}A)=\mu(A)$.
The dynamical stystem $(X,T)$ is called ergodic for $\mu$ if for every invariant set $A\in\B$, $\mu(A)=0$ or $\mu(X\setminus A)=0$.
It is called minimal if $A=\emptyset$ or $A=X$ whenever $A$ is a closed invariant set. It is equivalent to
every orbit being dense in $X$.
It is called uniquely ergodic if it has only one invariant ergodic probability measure.
Two topological dynamical systems $(X,T)$ and $(Y,S)$ are called conjugate if there exists a homeomorphism $\Psi:X\to Y$ such that
$\Psi\circ T=S\circ\Psi$.

\section{Dynamics of compatible maps in ${\ZZ}_p$}

In this section we give dynamical properties of compatible maps.
These are known results but we give the proofs for the sake of completeness and a better understanding of the next section.

\begin{prop}\label{isometry}
Let $f:{\ZZ}_p\to{\ZZ}_p$ be a compatible map. The following are equivalent :
\begin{enumerate}
\item $f$ is onto;
\item $f$ is an isometry;
\item for all $n\in \NN$, $\fn:{\ZZ}/p^n{\ZZ}\to{\ZZ}/p^n{\ZZ}$ (defined in Lemma \ref{compatible}) is bijective;
\item $f$ preserves the Haar measure.
\end{enumerate}
\end{prop}

\begin{proof}
{\bf $(1)\Longrightarrow (2)$} :
Let $f$ be onto. 
Let us show that for all $x\in \ZZ_p$ and all $n\in \NN$, $f(x+p^n{\ZZ}_p)=f(x)+p^n{\ZZ}_p$ (we already have
one inclusion because of the compatibility of $f$). 
By contradiction, suppose it is not true : 
there exist 
$x\in \ZZ_p$, $n\in \NN$, and $z\in \ZZ_p$ with $z\in f(x)+p^n{\ZZ}_p\setminus f(x+p^n{\ZZ}_p)$. 
Then $f(x)+p^n{\ZZ}_p=z+p^n{\ZZ}_p$. 
As $f$ is onto,
there exists $y\in \ZZ_p$ with $z=f(y)$ and $y\notin x+p^n{\ZZ}_p$.
This implies $(x+p^n{\ZZ}_p)\cap(y+p^n{\ZZ}_p)=\emptyset$. 
We can write

$$
{\ZZ}_p=(x+p^n{\ZZ}_p)\cup(y+p^n{\ZZ}_p)\cup\bigcup_{i\in\{0,\ldots,p^n-1\}, i\neq x, i\neq y}(i+p^n{\ZZ}_p)
$$

and, using the compatibility of $f$,

$$
f({\ZZ}_p)={\ZZ}_p\subset(z+p^n{\ZZ}_p)\cup\bigcup_{i\in\{0,\ldots,p^n-1\}, i\neq x, i\neq y}(f(i)+p^n{\ZZ}_p)
$$

which means that at least one ball of radius $p^{-n}$ is lacking in ${\ZZ}_p$.

{\bf $(2) \Longrightarrow (3)$} :
As $\ZZ/p^n \ZZ$ is finite, it suffices to show that $\fn$ is one-to-one.
Let $a,b$ be two different elements of $\ZZ /p^n \ZZ$ such that $\fn (a) =
\fn (b)$.
We can consider $a$ and $b$ as elements of $\ZZ_p$.
This implies that $f(a)-f(b) \in p^n \ZZ_p$.
Consequently, $\frac{1}{p^{n-1}} \leq \vert a-b \vert_p = \vert f(a) -f(b)
\vert_p \leq \frac{1}{p^n} $, a contradiction.
Hence $\fn$ is bijective.

{\bf $(3)\Longrightarrow (4)$} : Let us prove that for all $n\in{\NN}$, any $i\in{\ZZ}/p^n{\ZZ}$, we have
  $$f^{-1}(i+p^n{\ZZ}_p)=\fn^{-1}(i)+p^n{\ZZ}_p$$
where $\fn^{-1}(i)$ is reduced to one point as $\fn$ is bijective.
This will do the job as this will imply
 $$\mu_p(f^{-1}(i+p^n{\ZZ}_p))=\mu_p(\fn^{-1}(i)+p^n{\ZZ}_p)=\frac{1}{p^n}=\mu_p(i+p^n{\ZZ}_p)$$
First inclusion : using compatibility, $f(\fn^{-1}(i)+p^n{\ZZ}_p)\subset f(\fn^{-1}(i))+p^n{\ZZ}_p=i+p^n{\ZZ}_p$ and so
$\fn^{-1}(i)+p^n{\ZZ}_p\subset f^{-1}(f(\fn^{-1}(i)+p^n{\ZZ}_p))\subset f^{-1}(i+p^n{\ZZ}_p)$.

For the other inclusion, let $x\in f^{-1}(i+p^n{\ZZ}_p)$. We have $f(x)\in i+p^n{\ZZ}_p$. 
Let $x_n=\pi_n(x)$. Using compatibility, we have
$f(x)\in f(x_n+p^n{\ZZ}_p)\subset \fn(x_n)+p^n{\ZZ}_p$. 
Hence $f(x)$ is a point common to the two balls with the same radius
$\fn(x_n)+p^n{\ZZ}_p$ and $i+p^n{\ZZ}_p$. 
This implies that both balls have the same center, so $i=\fn(x_n)$ and $x\in
\fn^{-1}(i)+p^n{\ZZ}_p$.

{\bf $(4)\Longrightarrow (1)$} :
By contradiction, suppose that there exists $x\in{\ZZ}_p\setminus f({\ZZ}_p)$.
Let us show there exists $n$ such that
$x+p^n{\ZZ}_p\subset{\ZZ}_p\setminus f({\ZZ}_p)$. 
If it is not the case, for all $n\in \NN$ there exists $x_n\in(x+p^n{\ZZ}_p)\cap f({\ZZ}_p)$ and $y_n$ with $x_n=f(y_n)$ and
$\vert f(y_n)-x\vert_p\le p^{-n}$.
By compacity one can suppose that $(y_n)$ converges to some $y\in \ZZ_p$ and, thus, that
$x=f(y)$. 
This contradicts our assumption.
 
To end remark that $\mu_p(f({\ZZ}_p)) = \mu_p(f^{-1}(f({\ZZ}_p)))=\mu_p({\ZZ}_p)=1$
which is absurd because a ball $x+p^n{\ZZ}_p$ (which has a measure $p^{-n}$) is lacking in $f({\ZZ}_p)$.
\end{proof}

Now we deal with minimality of maps in ${\ZZ}_p$.
Let $T$ be the translation in ${\ZZ}_p$:
  $$\begin{array}{crcl}
     T: & {\ZZ}_p & \rightarrow & {\ZZ}_p \\
      & x & \mapsto & x+1
    \end{array}
$$

It is sometimes called odometers in base $(p^n)_n$ (\cite{Ku}).

\begin{lem}\label{minimality}
Let $n\in{\NN}$ and $g:{\ZZ}/n{\ZZ}\rightarrow{\ZZ}/n{\ZZ}$. Then $g$ is minimal if and only if:
  $$\left\{\begin{array}{l}
  \forall x\in{\ZZ}/n{\ZZ}, g^n(x)=x , \hbox{ and },\\
  \forall x\in{\ZZ}/n{\ZZ}, g^k(x)=x\Longrightarrow k\equiv 0\ [n] .
  \end{array}\right.$$
\end{lem}

\begin{proof}
The proof is easy and left to the reader.
\end{proof}

A {\it cycle} for a map $h : \ZZ /p^n \ZZ \to \ZZ /p^n \ZZ$ is a finite sequence $(x_k)_{0\leq k\leq K-1}$
such that for every $i\in\{0,\ldots,K-2\}$, $x_{i+1} = h (x_i)$, $h(x_{K-1})=x_0$ and $\# \{ x_k ; 0\leq k \leq K-1 \} = K$.
It is a {\it full-cycle} for $h$ if $K=p^n$.

\begin{theo}\label{equivalences}
Let $f:{\ZZ}_p\to{\ZZ}_p$ be an onto compatible map. 
The five following propositions are equivalent:
\begin{enumerate}
\item $f$ is minimal;
\item for all $n$, $\fn$ has a full-cycle in ${\ZZ}/p^n{\ZZ}$;
\item $f$ is conjugate to the translation $T(x)=x+1$ on $\ZZ_p$;
\item $f$ is uniquely ergodic;
\item $f$ is ergodic for the Haar measure.
\end{enumerate}
\end{theo}

\begin{proof}
{\bf $(1)\Longrightarrow (2)$}: First note that as $f$ is onto, for every $n$, $\fn$ is bijective. By contradiction, suppose that
$\fn$ does not have a full-cycle in ${\ZZ}/p^n{\ZZ}$ for some $n$. Then, as $\fn$ is bijective, it has several cycles and there exists a proper subset $A$ of
${\ZZ}/p^n{\ZZ}$ such that $\fn(A)=A$. Then $A+p^n\ZZ_p$ is a non empty closed invariant proper subset of $\ZZ_p$ and $f$ is not minimal.

{\bf $(2)\Longrightarrow (3)$}: We need to find a conjugacy $\Psi:{\ZZ}_p\rightarrow{\ZZ}_p$
such that
  $$
\Psi\circ f=T\circ\Psi=\Psi+1 .
$$
Actually, we will look for a compatible conjugacy $\Psi$. 
Hence we only need to define maps $\psi_n:{\ZZ}/p^n{\ZZ}\rightarrow{\ZZ}/p^n{\ZZ}$ such
that $\psi_n\circ\varphi_{n}=\varphi_{n}\circ\psi_{n+1}$ and $\psi_n \circ \fn = T_{/n} \circ \psi_n$.
Indeed, defining $\Psi_n=\pi_n^{-1}\circ\psi_n\circ\pi_n$ for any $n$, one can check that 
$(\Psi_n)_n$ uniformly converges to a conjugacy $\Psi$ we are looking for.

Let us fix $n\in{\NN}$. 
As $\fn:{\ZZ}/p^n{\ZZ}\rightarrow{\ZZ}/p^n{\ZZ}$ is onto and has only one cycle, it has a full-cycle, and 
thus $(\ZZ /p^n \ZZ , \fn)$ is minimal.
Consequently, it is sufficient to define $\psi_n$ on the orbit of 0,
which is a dense orbit. 
Notice that in this case, dense orbit only means that the orbit goes through each point of ${\ZZ}/p^n{\ZZ}$.

For all $n\in \NN$ and $k\in \{1,\ldots,p^n-1 \}$ we set 

$$
\psi_n (0) = 0 \hbox{ and } \psi_n(f_{/n}^k(0))=k .
$$ 

Hence, $\psi_n \circ \fn = T_{/n} \circ \psi_n$ for all $n\in \NN$.

Let us now check the compatibility condition $\psi_n\circ\varphi_{n}=\varphi_{n}\circ\psi_{n+1}$, that is to say, if $x=(x_n)_{n\in{\NN}}\in{\ZZ}_p$:

$$
\forall n\in{\NN}, \psi_{n+1}(x_{n+1})\equiv\psi_n(x_n)\ [p^n] .
$$

As $x_n$ and $x_{n+1}$ belong to, respectively, $\ZZ/p^n{\ZZ}$ and $\ZZ/p^{n+1}{\ZZ}$, 
there exist $0\leq k_n<p^n$ and $0\leq k_{n+1}<p^{n+1}$ such that $x_n=f_{/n}^{k_n}(0)$ and $x_{n+1}=f_{/n+1}^{k_{n+1}}(0)$. 
Consequently, $f_{/n+1}^{k_{n+1}}(0)\equiv f_{/n}^{k_n}(0)\ [p^n]$ because $\varphi_n(x_{n+1})=x_n$.
Moreover, $f$ being compatible, $f_{/n+1}^{k_{n+1}}(0)\equiv f_{/n}^{k_{n+1}}(0)\ [p^n]$.
This implies $f_{/n}^{k_{n+1}}(0)\equiv f_{/n}^{k_n}(0)\ [p^n]$ and, using Lemma \ref{minimality},
$k_{n+1}-k_n=0$ in ${\ZZ}/p^n{\ZZ}$ which is equivalent to
$\psi_{n+1}(f_{/n+1}^{k_{n+1}}(0))\equiv \psi_n(f_{/n}^{k_n}(0))\ [p^n]$, that is to say
$\psi_{n+1}(x_{n+1})\equiv\psi_n(x_n)\ [p^n]$.

{\bf $(3)\Longrightarrow (4)$}: By definition of the Haar measure.

{\bf $(4)\Longrightarrow (5)$}: Suppose that $f$ is uniquely ergodic. 
Using Proposition \ref{isometry}, 
$f$ being an isometry, the Haar measure $\mu_p$ is invariant.
Thus it is the only invariant ergodic probability measure.

{\bf $(5)\Longrightarrow (1)$}: 
Let $x,y$ in ${\ZZ}_p$ and $\epsilon >0$. 
It suffices to show that for some $k$, $|f^k(x) - y|_p < \epsilon$. 
Let $n$ such that
$p^{-n}<\varepsilon/2$. 
As $f$ is
ergodic for the Haar measure, there exists some $k>0$ such that $\mu_p(f^k(x+p^n{\ZZ}_p)\cap(y+p^n{\ZZ}_p))>0$. 
As $f$ is compatible $f^k(x+p^n{\ZZ}_p)\subset f^k(x)+p^n{\ZZ}_p$ and
thus $(f^k(x)+p^n{\ZZ}_p)\cap(y+p^n{\ZZ}_p)\neq\emptyset$. Therefore these two balls coincide and $f^k(x)$ and $y$
lie in the same ball of radius $p^{-n}$. Then $\vert f^k(x)-y\vert_p\le2p^{-n}\le\varepsilon$.
This concludes the proof.
\end{proof}

It is well known that minimal isometries on topological groups are conjugate to group rotations \cite{Ku}.
Here, we obtain a little more, namely  we prove that the conjugacy 
goes from $\ZZ_p$ to $\ZZ_p$.

\section{Minimality for polynomials}

\subsection{Some preliminary results}

In this section, we are interested in polynomials $f:{\ZZ}_p\to{\ZZ}_p$ with coefficients in ${\ZZ}_p$. 
For the sequel, we need some definitions inspired by \cite{DZ}.
Let $f:{\ZZ}_p\to{\ZZ}_p$ be a polynomial. 
Let $g_{n}=f^{p^n}$.

Suppose there exists a finite sequence $x_0,\ldots,x_{p^n-1}$ in $\ZZ_p$ so that $f_{/n}$ has a full cycle 
$\pi_n(x_0),\ldots,\pi_n (x_{p^n-1})$ in ${\ZZ}/p^n{\ZZ}$. 
Then, $g_n (x_0)$ belongs to $x_0 +p^n \ZZ_p$.
Consequently  $\frac{g_n(x_0)-x_0}{p^n}$ belongs to $\ZZ_p$.
Using Taylor's expansion, for $n\geq 1$, we get for all $z\in \ZZ_p$:

\begin{align}
\label{def-phi}
g_n (x_0+p^nz)\in 
x_0+p^n\left(\frac{g_n(x_0)-x_0}{p^n}\right)+p^n \left(g_n \right)'(x_0)z\ +p^{2n}\ZZ_p .
\end{align}

We set

\begin{align*}
\alpha_{n} (x_0)= \left(g_n \right)'(x_0), \ \beta_{n} (x_0) =\frac{g_n (x_0)-x_0}{p^n} \hbox{ and } \\
\begin{array}{llll}
\Phi_{n}(x_0) : & \ZZ_p & \to     & \ZZ_p \\
              & z     & \mapsto & \alpha_{n}(x_0)z +\beta_n(x_0).
\end{array}
\end{align*} 

We have

\begin{equation}
\label{rel-gn}
g_{n}(x_0+p^nz)
\in
x_0+p^n\Phi_{n}(x_0)(z) + p^{2n} \ZZ_p.
\end{equation}

Then,

\begin{align*}
p^n\beta_n (f(x_0)) & = g_n (f(x_0)) -f(x_0) = f(g_n (x_0)) -f(x_0)\\
& \equiv f(x_0  +p^n\Phi_{n}(x_0)(0) ) - f(x_0) \equiv p^n\Phi_{n}(x_0)(0) f'(x_0)\ [p^{2n}] \\
\end{align*}
 
Hence 

\begin{align}
\label{rel-betan}
\beta_n (f(x_0)) \equiv \beta_n (x_0) f'(x_0) \ [p^n].
\end{align}


\begin{lem}
\label{cycle}
Let $a,b\in \ZZ / p\ZZ$, and

$$
\begin{array}{lrcl}
  h : & {\ZZ}/p{\ZZ} & \to & {\ZZ}/p{\ZZ} \\
   & t & \mapsto & b+at
  \end{array}
$$
then $(\ZZ / p\ZZ , h)$ is minimal if and only if $a = 1$ and $b\not = 0$.
\end{lem}

\begin{proof}
The proof is easy and left to the reader.
\end{proof}

\begin{lem}
\label{lemme-equiv-min}
Let $f:{\ZZ}_p\to{\ZZ}_p$ be a polynomial and $n\geq 1$ be such that 
$(\ZZ / p^n \ZZ , f_{/n})$ is minimal.
Then, the following are equivalent :

\begin{enumerate}
\item
$(\ZZ / p^{n+1} \ZZ , f_{/n+1})$ is minimal;
\item
For all $x\in \ZZ_p$, $f^{p^{n}} (x) -x \not \in p^{n+1}\ZZ_p$ and $\left(f^{p^n}\right)'(x) \in 1 + p\ZZ_p$;
\item
There exists $x\in \ZZ_p$ such that $f^{p^{n}} (x) -x \not \in p^{n+1}\ZZ_p$ and $\left(f^{p^n}\right)'(x) \in 1 + p\ZZ_p$.
\end{enumerate}
\end{lem}

\begin{proof}
Suppose $(\ZZ / p^{n+1} \ZZ , f_{/n+1})$ is minimal.
Let $x\in \ZZ_p$.
If $f^{p^{n}} (x) -x $ belongs to $p^{n+1}\ZZ_p$,
then $f_{/n+1}$ has a cycle of length $p^{n}$, which is not possible. 
From \eqref{rel-gn} we see that, $(\ZZ / p^{n+1} \ZZ , f_{/n+1})$ being minimal, 
$\Phi_n (x)_{/1} : \ZZ/p\ZZ \to \ZZ /p\ZZ$ should be minimal.
Lemma \ref{cycle} implies $\left(f^{p^n}\right)'(x) \in 1 + p\ZZ_p$.
Hence (1) implies (2).
Suppose (2). 
Then, $\Phi_n (x)_{/1}$ is minimal for all $x\in \ZZ_p$.
We deduce with \eqref{rel-gn} that $(\ZZ / p^{n+1} \ZZ , f_{/n+1})$ is minimal :
(2) implies (1).
It is trivial that (2) gives (3). 
Let us show that (3) implies (2).
Suppose there exists $x\in \ZZ_p$ such that $f^{p^{n}} (x) -x \not \in p^{n+1}\ZZ_p$ and $\left(f^{p^n}\right)'(x) \in 1 + p\ZZ_p$.
Let $y\in \ZZ_p$.
Then, $(f^i_{/n}(y))_{0\leq i\leq p^n-1}$ is a full cycle in $\ZZ /p^n\ZZ$.
Hence,

$$
1\equiv\left(f^{p^n} \right)' (x) \equiv \prod_{i=0}^{p^n-1} f'\left(f^i (x)\right) \equiv \prod_{i=0}^{p^n-1} f'\left( i \right)
\equiv
\prod_{i=0}^{p^n-1} f'\left(f^i (y)\right)
\equiv
\left(f^{p^n} \right)' (y) \ [p^n] .
$$

We conclude using \eqref{rel-betan}.
\end{proof}

\begin{prop}\label{polynomials}
Let $f:{\ZZ}_p\to{\ZZ}_p$ be a polynomial, then
$(\ZZ_p , f)$ is minimal if and only if $({\ZZ}/p^\delta {\ZZ} ,f_{/ \delta} )$ is minimal, where 
$\delta = 2$ if $p> 3$ and $\delta= 3$ if $p\in \{ 2,3 \}$.
\end{prop}

\begin{proof}
The necessary condition is due to Theorem \ref{equivalences}.
Suppose $({\ZZ}/p^\delta{\ZZ} ,f_{/ \delta} )$ is minimal, i.e $f_{/ \delta }:{\ZZ}/p^\delta {\ZZ}\to{\ZZ}/p^\delta {\ZZ}$ has one full-cycle. 
From Theorem \ref{equivalences}, it is sufficient to prove that, for any $n$,
$\fn:{\ZZ}/p^n{\ZZ}\to{\ZZ}/p^n{\ZZ}$ has a full-cycle. 
Remark that, as the number of cycles in ${\ZZ}/p^n{\ZZ}$ is non decreasing with $n$, this implies that
$f_{/ n}$ has only one cycle for all $n \leq \delta$.

Let us proceed by induction on $n$ for the other cases.
Suppose $(\ZZ /p^n \ZZ , f_{/n})$ is minimal.
From Lemma \ref{lemme-equiv-min} it suffices to prove that 
$\beta_n (0) \not \in p\ZZ_p$ and $\alpha_n(0) \in 1 + p\ZZ_p$.
As $\delta$ is greater than $2$, Lemma \ref{lemme-equiv-min}
implies that $\alpha_{n-1} (0)$ belongs to $1 +p\ZZ_p$ and $\beta_{n-1} (0)$ does not belong to $p\ZZ_p$.
We set $\alpha_k =\alpha_k (0)$,  $\beta_k =\beta_k (0)$ and $\Phi_k = \Phi_k (0)$ for all $k$.
In order to conclude, we establish relations between $(\alpha_{n} , \beta_{n})$ and $(\alpha_{n-1} , \beta_{n-1})$.
Remark that  $g_{n+1}=g_{n}^p$ so, using \eqref{rel-gn},

\begin{align}
\label{calcul-derivee}
\alpha_{n} 
=
(g_{n-1}^p)'(0)
=
\prod_{i=0}^{p-1}g_{n-1}'(g_{n-1}^i(0))\in g_{n-1}'(0)^p +p^{n-1} \ZZ_p\subset 1 + p\ZZ_p .
\end{align}

For $\beta_n$ the situation is not so easy.
We need to consider three different cases.


Let us first consider the case $p=3$ and $n\geq 3$.
We have 

$$
\beta_{n}  =\frac{g_{n}(0)}{p^n}=\frac{g_{n-1}^p(0)}{p^n} , \ 
g_{n-1}^p(0)\in   \Phi_{n-1}^p (0)p^{n-1} +  p^{2(n-1)} \ZZ_p \hbox{ and }
$$

$$
\Phi_{n-1}^p (0) =  \beta_{n-1}\left(1+\alpha_{n-1}+\ldots+\alpha_{n-1}^{p-1}\right)\in \beta_{n-1}p+p^2\ZZ_p.
$$

Hence $\beta_n \in \beta_{n-1}  +p \ZZ_p$,
which conclude the proof for $p=3$.


Let us now consider the case $p=2$ and $n\geq 3$.
Then $n-2\geq 1$ and $(\ZZ /p^{n-2}\ZZ , f_{/n-2})$ is minimal.
From Lemma \ref{lemme-equiv-min} we obtain $\alpha_{n-2}=1+pz$ for some $z\in \ZZ_p$.
Consequently, proceeding as in \eqref{calcul-derivee},

\begin{align*}
\alpha_{n-1} 
&=
\alpha_{n-2} \alpha_{n-2} (g_{n-2} (0))= \alpha_{n-2}\left( \alpha_{n-2}  + p^{n-2}z' g_{n-2}'' (0)  \right) \\
& = 1+ 2pz +p^2z^2 + \alpha_{n-2} p^{n-2}z' g_{n-2}'' (0) \in 1 + p^2\ZZ_p ,
\end{align*}
for some $z'\in \ZZ_p$, because $g''_{n-2}(x)$ belongs to $p\ZZ_p$ for all $x\in \ZZ_p$.
Then, doing as in the previous case, we deduce that $\beta_n \in \beta_{n-1}  +p \ZZ_p$.


Let us end with the case $p\geq 5$ and $n\geq 2$.
Note that what we did in the case $p=3$ also holds here.
Hence we just have to consider the case $n=2$ and thus to prove that 
$\beta_2 \in \beta_1 +p\ZZ_p$.
We set $\gamma = g_{1}'' (0)/2$ (as $p\geq 5$, it belongs to $\ZZ_p$).
By induction, as observed in \cite{DZ}, we can prove that 

\begin{align*}
g_{1}^i (0)
& \equiv
p\beta_{1} \sum_{j=0}^{i-1} \alpha_{1}^j +p^2\gamma\beta_{1}^2  \sum_{j=0}^{i-2} \alpha_{1}^{i-2-j} \left(1+\alpha_{1}+\cdots + \alpha_{1}^j\right)^2 \ \ [p^3],
\end{align*}

hence

\begin{align}
\label{calc-iterate}
g_{1}^p (0)
& \equiv
p^2\beta_{1}  +p^2\gamma\beta_{1}^2  \frac{(p-1)p(2p-1)}{6} =  p^2\beta_{1}  \ \  [p^3].
\end{align}

This achieves the proof.
\end{proof}

\begin{rema}
The value of $\delta$ obtained in Proposition \ref{polynomials} cannot be smaller.

For $p=2$, let $f(x) = 1+3x+2x^3$. 
The map $f_{/2}$ has a full cycle 0,1,2,3 but $(\ZZ_2 , f)$ is not minimal : $f_{/3}$ has a cycle 0,1,6,3. 

For $p=3$, let $f(x) = 1+4x+4x^3+2x^5$.
The map $f_{/2}$ has a full cycle 0,1,2,6,7,5,3,4,8 but $(\ZZ_3 , f)$ is not minimal, $f_{/3}$ has a cycle 0,1,11,15,7,23,3,13,17.
\end{rema}


\begin{rema}\label{a0}
Let $f$ be a polynomial with $f(0)=a_0$. If $f$ is minimal, then $a_0\not\equiv 0\ [p]$, otherwise 0 is a fixed point modulo $p$ for $f$. Let $g(x)=\frac{1}{a_0}f(a_0x)$.
Then $g(0)=1$ and $(\ZZ_p,f)$ and $(\ZZ_p,g)$ are conjugate. Therefore, $f$ is minimal if and only if $g$ is minimal. Hence we will restrict in the following to polynomials
$f$ with $f(0)=1$. 
\end{rema}

\subsection{Characterization for $p=2$}

Using the previous ideas, we can characterize the polynomials that are minimal in ${\ZZ}_2$. The result is not new
(see for example \cite{La}) but the method described here is interesting as it will also apply to ${\ZZ}_3$.

Let $f(x) = a_d x^d +a_{d-1}x^{d-1} +\cdots + a_1x +1$ be a polynomial of degree $d$. 
We set 

$$
\begin{array}{ll}
A_0 = \sum_{i\in 2\ZZ, i\not =0} a_i, & A_1 = \sum_{i\in 1+2\ZZ} a_i.
\end{array}
$$

\begin{lem}
\label{lemme-z2z}
The dynamical system $(\ZZ / 2\ZZ , f_{/1})$ is minimal if and only if
 $$A_0+A_1\in 1+2\ZZ_2.$$
\end{lem}

\begin{proof}
If comes from the fact that $f_{/1}(0)=1$ and $f_{/1}(1)=1+A_0+A_1$. Hence $f_{/1}$ is minimal if and only if $1+A_0+A_1\equiv 0\ [2]$.
\end{proof}

\begin{lem}
\label{lem-min4}
The dynamical system $(\ZZ / 4\ZZ , f_{/2})$ is minimal if and only if 
  $$\left\{\begin{array}{l}
    a_1\equiv 1\ [2] \\
    A_1\equiv 1\ [2] \\
    A_0+A_1\equiv 1\ [4]
    \end{array}\right.$$
\end{lem}

\begin{proof}
Suppose $(\ZZ / 4\ZZ , f_{/2})$ minimal, then $(\ZZ / 2\ZZ , f_{/1})$ is also minimal, therefore $A_0+A_1\equiv 1\ [2]$ (see Lemma
\ref{lemme-z2z}). In view of Lemma \ref{lemme-equiv-min}, $(\ZZ / 4\ZZ , f_{/2})$ is minimal if and only if 
  $$A_0+A_1\equiv 1\ [2]\ ,\ (f^2)'(0)\in 1+2\ZZ_2\ \mbox{and}\ f^2(0)\in 2\ZZ_2\setminus 4\ZZ_2$$
First, $(f^2)'(0)\equiv f'(0)f'(1)\equiv a_1\sum_{i\geq 1}ia_i\equiv a_1A_1\ [2]$, which is equivalent to $a_1\in 1+2\ZZ_2$ and $A_1\in 1+2\ZZ_2$. Secondly, $f^2(0)=1+A_0+A_1$.
We thus get immediately $1+A_0+A_1\not\equiv 0\ [4]$ and consequently $A_0+A_1\equiv 1\ [4]$.
\end{proof}

\begin{theo}
The dynamical system $(\ZZ_2,f)$ is minimal if and only if 
  $$\left\{\begin{array}{l}
    a_1\equiv 1\ [2] \\
    A_1\equiv 1\ [2] \\
    A_0+A_1\equiv 1\ [4] \\
    2a_2+a_1A_1\equiv 1\ [4]
    \end{array}\right.$$
\end{theo}

\begin{proof}
Lemma \ref{lemme-equiv-min} implies that $(\ZZ_2,f)$ is minimal if and only if
  $$(\ZZ / 4\ZZ , f_{/2})\ \mbox{is minimal,}\ (f^4)'(0)\in 1+2\ZZ_2\ \mbox{and}\ f^4(0)\in 4\ZZ_2\setminus 8\ZZ_2.$$
First, $(f^4)'(0)\equiv ((f^2)'(0))^2\ [2]$ so this condition has already been checked. Secondly, using Taylor's expansion, we get
  $$f^2(pz)\equiv p(\beta_1+\alpha_1z)+p^2z^2\frac{(f^2)''(0)}{2}\ [8]$$
and then
  $$f^4(0)\equiv 2\beta_1(1+(f^2)'(0))+4\beta_1^2\frac{(f^2)''(0)}{2}\ [8]$$
with $\beta_1\equiv 1 [2]$ and $1+(f^2)'(0)\equiv 0\ [2]$. The condition $f^4(0)\in 4\ZZ_2\setminus 8\ZZ_2$ is therefore equivalent to
  $$(f^2)'(0)+(f^2)''(0)\equiv 1\ [4].$$
Using $a_1\equiv 1\ [2]$,
 \begin{eqnarray*}
 (f^2)'(0) & \equiv & a_1\sum_{i\geq 1}ia_i\equiv a_1\left(\sum_{i\geq 0}a_{4i+1}+2\sum_{i\geq 0}a_{4i+2}-\sum_{i\geq 0}a_{4i+3}\right)\ [4]\\
  & \equiv & a_1\sum_{i\geq 0}a_{4i+1}+2\sum_{i\geq 0}a_{4i+2}-a_1\sum_{i\geq 0}a_{4i+3}\ [4].
  \end{eqnarray*}
Moreover, as $\sum_{i\geq 2}i(i-1)a_i\equiv 0\ [2]$, $a_1^2\equiv 1\ [2]$ and $\sum_{i\geq 1}ia_i\equiv A_1\equiv 1\ [2]$:
  \begin{eqnarray*}
  (f^2)''(0) & = & f''(0)f'(1)+(f'(0))^2f''(1)=2a_2\sum_{i\geq 1}ia_i+a_1^2\sum_{i\geq 2}i(i-1)a_i \\
   & \equiv & 2a_2+\sum_{i\geq 2}i(i-1)a_i\equiv\sum_{i\geq 3}i(i-1)a_i\ [4]\\
   & \equiv & 2\sum_{i\geq 1}a_{4i+2}+2\sum_{i\geq 0}a_{4i+3}\ [4].
   \end{eqnarray*}
This gives, using $2-a_1\equiv a_1\ [4]$
  \begin{eqnarray*}
  (f^2)'(0)+(f^2)''(0) & \equiv & 2a_2+a_1\sum_{i\geq 0}a_{4i+1}+(2-a_1)\sum_{i\geq 0}a_{4i+3}\ [4] \\
  & \equiv & 2a_2+a_1A_1\ [4].
  \end{eqnarray*}
This achieves the proof.
\end{proof}

These conditions are easily proved to be equivalent to those given in \cite{La}:
  $$\left\{\begin{array}{l}
    a_1\equiv 1\ [2] \\
    A_1-a_1\equiv 2a_2\ [4] \\
    A_0-a_2\equiv a_1+a_2-1\ [4].
    \end{array}\right.$$

\begin{coro}
Let $f(x) = a_d x^d +a_{d-1}x^{d-1} +\cdots + a_1x +a_0$ be a polynomial of degree $d$. The dynamical system $(\ZZ_2,f)$ is minimal if and only if 
  $$\left\{\begin{array}{l}
    a_1\equiv 1\ [2] \\
    A'_1\equiv 1\ [2] \\
    A'_0+A'_1\equiv 1\ [4] \\
    2a_2a_0+a_1A'_1\equiv 1\ [4]
    \end{array}\right.$$
where 
  $$A'_0 = \sum_{i\in 2\ZZ, i\not =0} a_ia_0^{i-1}\ \ and\ \  A'_1 = \sum_{i\in 1+2\ZZ} a_ia_0^{i-1}.$$
\end{coro}

\subsection{Characterization for $p=3$}

Let $f(x) = a_d x^d +a_{d-1}x^{d-1} +\cdots + a_1x +1$ be a polynomial of degree $d$. 
We set 

$$
\begin{array}{ll}
A_0 = \sum_{i\in 2\ZZ, i\not =0} a_i, & A_1 = \sum_{i\in 1+2\ZZ} a_i , \\
D_0 = \sum_{i\in 2\ZZ, i\not =0} ia_i, & D_1 = \sum_{i\in 1+2\ZZ} ia_i .
\end{array}
$$

\begin{lem}
\label{lemme-z3z}
The dynamical system $(\ZZ / 3\ZZ , f_{/1})$ is minimal if and only if 

$$
A_0 \in 3\ZZ_3 \hbox{ and } A_1 \in 1+3\ZZ_3 .
$$
\end{lem}

\begin{proof}
Suppose $f$ is minimal.
We remark that if $x\not =0$ then $f_{/1}(x) = 1+A_0+A_1x$ and $f_{/1}(0)=1$.
Hence $A_0+A_1x\not\equiv 0 \ [3]$ for all $x\not \equiv 0\ [3]$.
Then, $A_0\equiv 0\ [3]$ and $A_1 \equiv 1\ [3]$.
Consequently $(\ZZ /3\ZZ, f_{/1})$ is minimal.
The reciprocal is as easy to establish.
\end{proof}

\begin{lem}
\label{lem-min9}
The dynamical system $(\ZZ / 9\ZZ , f_{/2})$ is minimal if and only if $A_0 \in 3\ZZ_3 $,
$A_1 \in 1+3\ZZ_3$ and $f$ fulfils one of the following conditions:

\begin{enumerate}
\item 
$D_0 \in 3\ZZ_3 $,
$D_1 \in 2+3\ZZ_3$, $a_1 \in 1+3\ZZ_3$ and 
$A_1+5 \not \in 9\ZZ_3$;
\item
$D_0 \in 3\ZZ_3 $,
$D_1 \in 1+3\ZZ_3$, $a_1 \in 1+3\ZZ_3$ and 
$A_0+6 \not \in 9\ZZ_3$;
\item 
$D_1 \in 3\ZZ_3 $,
$D_0 \in 1+3\ZZ_3$, $a_1 \in 2+3\ZZ_3$ and 
$A_1+5\not \in 9\ZZ_3  $;
\item
$D_1 \in 3\ZZ_3 $,
$D_0 \in 2+3\ZZ_3$, $a_1 \in 2+3\ZZ_3$ and 
$A_0+6 \not \in 9\ZZ_3$;
\end{enumerate}
\end{lem}

\begin{proof}
In Lemma \ref{lemme-equiv-min} we prove $(\ZZ / 9\ZZ , f_{/2})$ is minimal 
if and only if 
\begin{equation*}
  (\ZZ /3\ZZ , f_{/1} )\ \hbox{is minimal,}\  (f^3)' (0) \in 1+3\ZZ_3\ \hbox{and}\   f^3 (0) \in 3\ZZ_3 \setminus 9\ZZ_3.
\end{equation*}

Suppose $(\ZZ /9\ZZ , f_{/2} )$ is minimal. From Lemma \ref{lemme-z3z}, $A_0 \in 3 \ZZ_3$ and $A_1\in 1+3\ZZ_3$.
We easily check that $(f^3)' (0) \equiv f'(0)f'(1)f'(2)\equiv a_1 (D_1^2 - D_0^2) \ [3]$.
Hence either $D_0\in 3\ZZ_3$ and $D_1\not\in 3\ZZ_3$ or $D_1\in 3\ZZ_3$ and $D_0\not\in 3\ZZ_3$.
Moreover we have $f(0)=1$, $f^2 (0)=1+A_0 + A_1$ and 

\begin{align*}
f^3 (0) \equiv & 1+\sum_{i=1}^da_i (1+A_0+A_1)^i \equiv  1+\sum_{i=1}^d a_i \left((1+A_1)^i + iA_0(1+A_1)^{i-1} \right) & [9] \\
\equiv &   1+\sum_{i=1}^da_i \left(  2^i +i2^{i-1}(A_1-1)+iA_0 2^{i-1} \right) &  [9] \\
\equiv & 1+\sum_{i=1}^da_i \left(  (-1)^i + 3i(-1)^{i-1} +i(-1)^{i-1}(A_0+A_1-1) \right) &  [9] \\
\equiv & 1+ A_0-A_1 +(A_0+A_1+2)(D_1-D_0) &  [9] .
\end{align*}

Consequently, $(\ZZ / 9\ZZ , f_{/2})$ is minimal 
if and only if $A_0 \in 3 \ZZ_3$, $A_1\in 1+3\ZZ_3$, $a_1 (D_1^2 - D_0^2)\in 1+3\ZZ_3$ and $1+ A_0-A_1 +(A_0+A_1+2)(D_1-D_0)\in 3\ZZ_3 \setminus 9\ZZ_3$.

Suppose $D_0 \in 3\ZZ_3$. 
Then $a_1 \in 1+3\ZZ_3$ and

\begin{align*}
f^3 (0) \equiv  \left\{ \begin{array}{lll} 3+2A_0 & [9] & \hbox{ if } D_1\in 1+3\ZZ_3,\\  5+A_1 & [9] & \hbox{ if } D_1\in 2+3\ZZ_3. \end{array} \right. 
\end{align*}

Suppose $D_1 \in 3\ZZ_3$.
Then $a_1 \in 2+3\ZZ_3$ and

\begin{align*}
f^3 (0) \equiv \left\{ \begin{array}{lll} -1-2A_1 & [9] & \hbox{ if } D_0\in 1+3\ZZ_3,\\  -6-A_0 & [9] & \hbox{ if } D_0\in 2+3\ZZ_3. \end{array} \right. 
\end{align*}
This achieves the proof.
\end{proof}

\begin{theo}
The dynamical system $(\ZZ_3 , f)$ is minimal if and only if $A_0 \in 3\ZZ_3 $,
$A_1 \in 1+3\ZZ_3$ and $f$ fulfils one of the conditions (1), (2), (3) or (4):

$$
\begin{array}{|l|l|l|l|l|l|}
\hline
           & D_0 \ [3] & D_1 \ [3] & a_1\ [3]  &                \\
\hline
\hbox{(1)} & 0 & 2 & 1 & \begin{array}{l} A_1+5 \not\equiv 0 \ [9] \\ A_1+5 \not\equiv 3a_2 +3 \sum_{j\geq 0}a_{5+6j} \ [9]\end{array} \\
\hline
\hbox{(2)} & 0 & 1 & 1 & \begin{array}{l} A_0+6 \not\equiv 0 \ [9] \\ A_0+6 \not\equiv 6a_2 +3 \sum_{j\geq 0}a_{2+6j} \ [9]\end{array} \\
\hline
\hbox{(3)} & 1 & 0 & 2 & \begin{array}{l} A_1+5 \not\equiv 0 \ [9] \\ A_1+5 \not\equiv 6a_2 +3 \sum_{j\geq 0}a_{5+6j} \ [9]\end{array}\\
\hline
\hbox{(4)} & 2 & 0 & 2 & \begin{array}{l} A_0+6 \not\equiv 0 \ [9] \\ A_0+6 \not\equiv 3a_2 +3 \sum_{j\geq 0}a_{2+6j} \ [9]\end{array}\\
\hline
\end{array}
$$
\end{theo}

\begin{proof}
From Lemma \ref{lemme-equiv-min} and Proposition \ref{polynomials}, we know that $(\ZZ_3 , f)$ is minimal if and only if $(\ZZ / 9\ZZ , f_{/2})$
is minimal, $\left(f^9\right)'(0) \in 1+3\ZZ_3$ and $f^{9} (0) \not \in 3^3 \ZZ_3$.
From Lemma \ref{lemme-equiv-min}, as $(\ZZ / 9\ZZ , f_{/2})$ is minimal, $\left(f^3 \right)' (0) \in 1+3\ZZ_3$.
The calculus \eqref{calcul-derivee} shows that $\left(f^9\right)'(0) \in 1+3\ZZ_3$.
Hence, using Lemma \ref{lem-min9}, $(\ZZ_3 , f)$ is minimal if and only if $f$ fulfils Conditions (1), (2), (3), (4), of Lemma \ref{lem-min9}, and 
$f^{9} (0) \not \in 3^3 \ZZ_3$.
But from \eqref{calc-iterate} we see that 

$$
f^9 (0) \equiv 3^2\beta_1 +3^2\gamma \beta_1^2 5 =3^2(\beta_1 -\gamma ) \ [3^3] ,
$$

where $\beta_1 = \frac{f^3 (0)}{3}$ and $\gamma = \frac{\left( f^3 \right)'' (0)}{2}$.
Thus, $f^{9} (0) \not \in 3^3 \ZZ_3$ if and only if $\beta_1 \not\equiv \gamma \ [3]$.
That is to say 

$$
2f^3 (0) \not\equiv 3\left(f^3 \right)'' (0) \ [9].
$$

We already computed $f^3 (0)$ in the proof of Lemma \ref{lem-min9}.
Let us compute $\left( f^3 \right)'' (0)$ :

\begin{align*}
\left( f^3 \right)'' (0) \equiv & f'(f^2(0))f'(1)f''(0)+ f'(f^2(0))f''(1)f'(0)^2+ f''(f^2(0))f'(1)^2f'(0)^2 \\
\equiv& f'(1+A_0+A_1)f'(1)2a_2+ f'(1+A_0+A_1)f''(1)+ f''(1+A_0+A_1) \ [3] \\
\equiv& -f'(-1)f'(1)a_2+ f'(-1)f''(1)+ f''(-1) \ [3] \\
\equiv& -(D_1^2-D_0^2)a_2+ (D_1-D_0)f''(1)+   f''(-1) \ [3] \\
\end{align*}

Moreover, we easily obtain that 

$$
f''(1) = -\sum_j a_{2+3j} \hbox{ and } f''(-1) = -\sum_j a_{2+3j}(-1)^j . 
$$

This gives:

$$
\begin{array}{|l|l|l|}
\hline
                       & 2f^3 (0)\ [9] & 3\left(f^3 \right)'' (0)\ [9]  \\
\hline
\hbox{ (1) Lemma \ref{lem-min9}} & 2A_1+1 & 6a_2 +6 \sum_{j\geq 0}a_{5+6j} \\
\hline
\hbox{ (2) Lemma \ref{lem-min9}} & 4A_0+6 & 6a_2 +3 \sum_{j\geq 0}a_{2+6j} \\
\hline
\hbox{ (3) Lemma \ref{lem-min9}} & 2A_1+1 & 3a_2 +6 \sum_{j\geq 0}a_{5+6j} \\
\hline
\hbox{ (4) Lemma \ref{lem-min9}} & 4A_0+6 & 3a_2 +3 \sum_{j\geq 0}a_{2+6j} \\
\hline
\end{array}
$$

which gives the conditions :

$$
\begin{array}{|l|l|}
\hline
                       &  \hbox{Condition} \\
\hline
\hbox{ (1) Lemma \ref{lem-min9}} & A_1+5 \not\equiv 3a_2 +3 \sum_{j\geq 0}a_{5+6j} \ [9] \\
\hline
\hbox{ (2) Lemma \ref{lem-min9}} & A_0+6 \not\equiv 6a_2 +3 \sum_{j\geq 0}a_{2+6j} \ [9]\\
\hline
\hbox{ (3) Lemma \ref{lem-min9}} & A_1+5 \not\equiv 6a_2 +3 \sum_{j\geq 0}a_{5+6j} \ [9]\\
\hline
\hbox{ (4) Lemma \ref{lem-min9}} & A_0+6 \not\equiv 3a_2 +3 \sum_{j\geq 0}a_{2+6j} \ [9]\\
\hline
\end{array}
$$

This achieves the proof.
\end{proof}

\begin{coro}
Let $f:{\ZZ}_3\to{\ZZ}_3$ be a polynomial defined by 
$f(x)=a_dx^d+a_{d-1}x^{d-1}+\cdots +a_1x+a_0$.
The dynamical system $(\ZZ_3 , f)$ is minimal if and only if $A'_0 \in 3\ZZ_3 $,
$A'_1 \in 1+3\ZZ_3$ and $f$ fulfils one of the Conditions (1), (2), (3) or (4):

$$
\begin{array}{|l|l|l|l|l|l|}
\hline
           & D'_0 \ [3] & D'_1 \ [3] & a_1\ [3]  &                \\
\hline
\hbox{(1)} & 0 & 2 & 1 & \begin{array}{l} A'_1+5 \not\equiv 0 \ [9] \\ A'_1+5 \not\equiv 3a_2a_0 +3 \sum_{j\geq 0}a_{5+6j}a_0^{4+6j} \ [9]\end{array} \\
\hline
\hbox{(2)} & 0 & 1 & 1 & \begin{array}{l} A'_0+6 \not\equiv 0 \ [9] \\ A'_0+6 \not\equiv 6a_2a_0 +3 \sum_{j\geq 0}a_{2+6j}a_0^{1+6j} \ [9]\end{array} \\
\hline
\hbox{(3)} & 1 & 0 & 2 & \begin{array}{l} A'_1+5 \not\equiv 0 \ [9] \\ A'_1+5 \not\equiv 6a_2a_0 +3 \sum_{j\geq 0}a_{5+6j}a_0^{4+6j} \ [9]\end{array}\\
\hline
\hbox{(4)} & 2 & 0 & 2 & \begin{array}{l} A'_0+6 \not\equiv 0 \ [9] \\ A'_0+6 \not\equiv 3a_2a_0 +3 \sum_{j\geq 0}a_{2+6j}a_0^{1+6j} \ [9]\end{array}\\
\hline
\end{array}
$$

where

$$
\begin{array}{ll}
A'_0 = \sum_{i\in 2\ZZ, i\not =0} a_ia_0^{i-1}, & A'_1 = \sum_{i\in 1+2\ZZ} a_ia_0^{i-1} , \\
D'_0 = \sum_{i\in 2\ZZ, i\not =0} ia_ia_0^{i-1}, & D'_1 = \sum_{i\in 1+2\ZZ} ia_ia_0^{i-1} .
\end{array}
$$

\end{coro}

\begin{coro}\label{cndegre5z3}
Let $f:{\ZZ}_3\to{\ZZ}_3$ be a polynomial defined by 
$f(x)=a_5x^5+a_4x^4+a_3x^3+a_2x^2+a_1x+1$.
Then, $(\ZZ_3 , f)$ is minimal if and only $f$ satisfies one of the following conditions

$$
\begin{array}{|l|l|l|l|l|l|l|}
\hline
           & a_1 \ [3] & a_2 \ [3] & a_3\ [3]  & a_4\ [3] & a_5\ [3] &        \\
\hline
\hbox{(1)} & 1 & 0 & 1 & 0 & 2 & a_1+a_3 +a_5 \equiv 7 \ [9]  \\
\hline
\hbox{(2)} & 1 & 0 & 0 & 0 & 0 & a_2 +a_4 \equiv 0 \hbox{ or } 6 \ [9] \\
\hline
\hbox{(3)} & 2 & 1 & 0 & 2 & 2 & a_1+a_3 +a_5 \equiv 7 \ [9]\\
\hline
\hbox{(4)} & 2 & 2 & 0 & 1 & 2 & a_2 + a_4 \equiv 0 \ [9]\\
\hline
\end{array}
$$

\end{coro}

\subsection{Discussion}
We had planned to address the question of minimality of polynomial dynamics in $\ZZ_p$ for any prime $p$. We have no idea if this is tractable or not to get a
general condition for all prime numbers $p>3$. We have obtained sufficient conditions for any $p$, but they are far from necessary, as we have seen 
performing simulations. In some sense, it is not surprising since giving conditions on the coefficients to ensure bijectivity of polynomials is already known to be 
a difficult problem \cite{LN}. 

\section*{Acknowledgement}
We would like to thank the mathematical department of Vaxj\"o University, Sweden, where part of this
work has been initiated, for its hospitality. We are grateful to Prof. Anashin, Prof. Fan and Prof. Khrennikov for fruitful discussions



\begin{thebibliography}{MMMM}

\bibitem{An1}V. Anashin,
{\it Uniformly distributed sequences of p-adic integers II},
Discrete Math. Appl. {\bf 12} (2002), no. 6, 527-590.

\bibitem{An2}V. Anashin,
{\it Uniformly distributed sequences in computer algebra or how to construct program generators of random numbers. Computing mathematics and cybernetics, 2},
J. Math. Sci.{\bf 89} (1998) no. 4, 1355-1390.

\bibitem{BS}J. Bryk and C.E. Silva,
{\it Measurable dynamics of simple p-adic polynomials},
Amer. Math. Monthly {\bf 112} (2005) no. 3, 212-232.

\bibitem{CP}Z. Coelho and W. Parry,
{\it Ergodicity of p-adic multiplications and the distribution of Fibonacci numbers},
Amer. Math. Soc. Transl. Ser. 2, {\bf 202} (2001) American Mathematical Society, 51-70.

\bibitem{DZ}D.L. Desjardins and M.E. Zieve,
{\it On the structure of polynomial mappings modulo an odd prime power},
Preprint, http://fr.arXiv.org/abs/math.NT/0103046.

\bibitem{FLWZ} A-H. Fan, L. Liao, Y-F. Wang, D. Zhou,
{\it $p$-adic repellers in $\QQ_p$ are subshifts of finite type},
C. R. Math. Acad. Sci. Paris {\bf 344} (2007), no. 4, 219-224.

\bibitem{FLYZ} A-H. Fan, M-T. Li, J-Y. Yao, D. Zhou,
{\it Strict ergodicity of affine $p$-adic dynamical systems on $\ZZ_p$},
Adv. Math. {\bf 214} (2007), no. 2, 666-700.

\bibitem{KN} A. Y. Khrennikov and M. Nilsson,
{\it $P$-adic deterministic and random dynamics},
Mathematics and Its Applications {\bf 574}, Kluwer Academic Publishers, 2004.

\bibitem{Kn} D. E. Knuth,
{\it The art of computer programming, Vol. 2 : Seminumerical algorithms (3rd edition)},
Addison-Wesley Publ. Co, 1998.

\bibitem{Ku} P. Kurka,
{\it Topological and symbolic dynamics},
Cours sp\'ecialis\'es, Collection SMF, 2003.

\bibitem{La}M. V. Larin,
{\it Transitive polynomial transformations of residue class rings},
Discrete Math. Appl. Vol {\bf 12} (2002) no. 2, 127-140.

\bibitem{LN}R. Lidl, H. Niederreiter,
{\it Finite fields},
Encyclopedia of Mathematics and its Applications, 20. Cambridge University Press, Cambridge, 1997.

\bibitem{Ro} A.M. Robert,
{\it A Course in $p$-adic Analysis},
Graduate Texts in Mathematics 198, Springer, 2000.


\end{thebibliography}
\end{document}